\documentclass[11pt,leqno]{amsart}

\usepackage{latexsym}
\usepackage{amssymb}
\usepackage{amsmath}
\usepackage{amsthm}
\usepackage{verbatim}
\usepackage{pdfsync}
\usepackage[margin=1.2in]{geometry}

\usepackage{color}
\usepackage{tikz}
\usepackage[all]{xy}

\numberwithin{equation}{section}  

\begin{document}



\newcommand{\zxz}[4]{\begin{pmatrix} #1 & #2 \\ #3 & #4 \end{pmatrix}}
\newcommand{\abcd}{\zxz{a}{b}{c}{d}}
\newcommand{\kzxz}[4]{\left(\begin{smallmatrix} #1 & #2 \\ #3 &
#4\end{smallmatrix}\right) }
\newcommand{\kabcd}{\kzxz{a}{b}{c}{d}}




\newcommand{\A}{{\mathbb A}}
\newcommand{\C}{{\mathbb C}}
\newcommand{\F}{{\mathbb F}}
\newcommand{\G}{{\mathbb G}}
\newcommand{\R}{{\mathbb R}}
\newcommand{\Q}{{\mathbb Q}}
\newcommand{\X}{{\mathbb X}}
\newcommand{\Z}{{\mathbb Z}}
\newcommand{\HZ}{\widehat{\Z}}


\newcommand{\rom}[1]{\text{\rm #1}}
\renewcommand{\roman}{\rm}

\newcommand{\Aut}{\text{\rm Aut}}
\newcommand{\CH}{\widehat{\text{\rm CH}}}
\newcommand{\cha}{{\text{\rm char}}}
\newcommand{\CHe}{\text{\rm CHeeg}}
\newcommand{\degh}{\widehat{\text{\rm deg}}}
\newcommand{\degH}{\widehat{\text{\rm deg}}}    
\newcommand{\diag}{{\text{\rm diag}}}
\newcommand{\Diff}{\text{\rm Diff}}
\newcommand{\disc}{\text{\rm discr}}
\renewcommand{\div}{\text{\rm div}}
\newcommand{\divh}{\widehat{\text{\rm div}}}
\newcommand{\DS}{\text{\rm DS}}
\newcommand{\Ei}{\text{\rm Ei}}
\newcommand{\End}{\text{\rm End}}
\newcommand{\ev}{{\text{\rm ev}}}
\newcommand{\Gal}{\text{\rm Gal}}
\newcommand{\GL}{\text{\rm GL}}
\newcommand{\GSpin}{\text{\rm GSpin}}
\newcommand{\Hom}{\text{\rm Hom}}
\newcommand{\hor}{{\text{\rm horiz}}}
\newcommand{\id}{\text{\rm id}}
\newcommand{\im}{\text{\rm im}}
\renewcommand{\Im}{\text{\rm Im}}
\newcommand{\inv}{{\text{\rm inv}}}
\newcommand{\Jac}{\text{\rm Jac}}
\newcommand{\Leray}{{\mathrm L}}
\newcommand{\Lie}{\text{\rm Lie}}
\newcommand{\Mp}{\text{\rm Mp}}
\newcommand{\mult}{\text{\rm mult}}
\newcommand{\MW}{\text{\rm MW}}
\newcommand{\MWt}{\widetilde{\MW}}
\newcommand{\new}{\text{\rm new}}
\newcommand{\Nm}{\text{\rm Nm}}
\newcommand{\ord}{\text{\rm ord}}
\newcommand{\PGL}{\text{\rm PGL}}
\newcommand{\Pic}{\text{\rm Pic}}
\newcommand{\Pich}{\widehat{\text{\rm Pic}}}
\newcommand{\pr}{\text{\rm pr}}
\newcommand{\ra}{\text{\rm ra}}
\newcommand{\Rao}{\mathrm R}
\renewcommand{\Re}{\text{\rm Re}}
\newcommand{\sgn}{\text{\rm sgn}}
\newcommand{\sig}{\text{\rm sig}}
\newcommand{\SL}{\text{\rm SL}}
\newcommand{\SO}{\text{\rm SO}}
\newcommand{\Sp}{\text{\rm Sp}}
\newcommand{\Spec}{\text{\rm Spec}\, }
\newcommand{\Spf}{\text{\rm Spf}}
\newcommand{\supp}{\text{\rm supp}}
\newcommand{\Sym}{{\text{\rm Sym}}}
\newcommand{\tr}{\text{\rm tr}}
\newcommand{\type}{\text{\rm type}}
\newcommand{\Ver}{\text{\rm Vert}}
\newcommand{\vol}{\text{\rm vol}}
\newcommand{\Wald}{\text{\rm Wald}}


\newcommand{\Cal}{\mathcal}     

\newcommand{\AHH}{\hat{\Cal A}}   
\newcommand{\CHH}{\hat{\Cal C}}
\newcommand{\MM}{\Cal D}          
\newcommand{\MMb}{\MM^\bullet}
\newcommand{\ssplit}{\text{\bf split}}
\newcommand{\whcc}{\widehat{\Cal C}}
\newcommand{\CO}{\mathcal O}
\newcommand{\COH}{\widehat{\CO}}
\newcommand{\M}{\Cal M}
\newcommand{\OB}{\Cal O_B}
\newcommand{\XX}{\mathcal X}
\newcommand{\bXX}{\bar\XX}
\newcommand{\wc}{\hat{\Cal C}}
\newcommand{\wch}{\wc^{\text{\rm hor}}}
\newcommand{\ZZ}{\Cal Z}
\newcommand{\ZH}{\widehat{\Cal Z}}   
\newcommand{\Zh}{\widehat{\Cal Z}}
\newcommand{\ZZh}{\ZZ^{\text{\rm hor}}}
\newcommand{\ZZv}{\ZZ^{\text{\rm ver}}}
\newcommand{\ZZhh}{\Zh^{\text{\rm hor}}}
\newcommand{\ZZhv}{\Zh^{\text{\rm ver}}}


\newcommand{\nass}{\noalign{\smallskip}}
\newcommand{\snass}{\noalign{\vskip 2pt}}
\newcommand{\tent}[1]{ \vphantom{\vbox to #1pt{}} }   


\newcommand{\scr}{\scriptstyle}
\newcommand{\disp}{\displaystyle}

\font\cute=cmitt10 at 12pt
\font\smallcute=cmitt10 at 9pt
\newcommand{\kay}{{\text{\cute k}}}
\newcommand{\smallkay}{{\text{\smallcute k}}}

\renewcommand{\a}{\alpha}
\renewcommand{\b}{\beta}
\newcommand{\e}{\epsilon}
\renewcommand{\l}{\lambda}
\renewcommand{\L}{\Lambda}
\renewcommand{\o}{\omega}
\renewcommand{\O}{\Omega}
\renewcommand{\P}{\Phi}
\newcommand{\ph}{\varphi}
\newcommand{\phih}{\widehat{\phi}}
\newcommand{\wphi}{\widehat{\phi}}
\newcommand{\phit}{\widetilde{\phi}}
\newcommand{\s}{\sigma}
\newcommand{\vth}{\vartheta}


%

\newcommand{\Pt}{P}
\newcommand{\Ph}{\P}
\newcommand{\Pht}{\tilde \P}   
\newcommand{\Kt}{K}           
\newcommand{\Mt}{M}

\newcommand{\pht}{\widetilde{\phi}}
\newcommand{\It}{I}
\newcommand{\Jt}{\widetilde{J}}
\newcommand{\lt}{\widetilde{\l}}
\newcommand{\vp}{\varpi}

\newcommand{\bom}{{\boldsymbol{\o}}}
\newcommand{\hbom}{\widehat{\bom}}
\newcommand{\ff}{{\bold f}}
\newcommand{\fsp}{\boldsymbol{f}_{\!\rm sp}}
\newcommand{\fev}{\boldsymbol{f}_{\!\rm ev}}
\newcommand{\fb}{\boldsymbol{f}}
\newcommand{\J}{\und{J}'}
\newcommand{\JJ}{\bold J'}
\newcommand{\V}{\bold V}
\newcommand{\xx}{\bold x}

\newcommand{\g}{{\mathfrak g}}
\renewcommand{\H}{\mathfrak H}


\newcommand{\back}{\backslash}
\newcommand{\CT}[1]{\operatornamewithlimits{CT}_{#1}}
\renewcommand{\d}{\partial}
\newcommand{\db}{\bar\partial}
\newcommand{\dbar}{\bar{\partial}}
\newcommand{\gs}[2]{\langle \,#1,#2\,\rangle}
\newcommand{\Gt}{G}
\newcommand{\hfal}{h_{\text{\rm Fal}}}
\newcommand{\II}{\int^\bullet}
\newcommand{\isoarrow}{\ {\overset{\sim}{\longrightarrow}}\ }
\newcommand{\lisoarrow}{\ {\overset{\sim}{\longleftarrow}}\ }
\newcommand{\limdir}[1]{\underset{\underset{#1}{\rightarrow}}{\lim}}
\newcommand{\lan}{\operatorname{\langle}\hskip .5pt}
\newcommand{\ran}{\,\operatorname{\rangle}}
\newcommand{\lra}{\longrightarrow}
\newcommand{\doublelra}{\ {\overset{\scr\lra}{\scr\lra}}\ }
\newcommand{\nat}{\natural}
\newcommand{\notmid}{\mkern-5mu\not\mkern5mu\mid}
\newcommand{\Optoc}{\text{\rm Opt}(O_{c^2d},O_B)}
\newcommand{\psim}{\psi^{-}}
\newcommand{\qeq}{\ \overset{??}{=}\ }
\newcommand{\sh}{\sharp}
\newcommand{\thCH}{\theta^{\text{\rm ar}}}
\newcommand{\wht}{\widehat{\theta}}     
\newcommand{\triv}{1\!\!1}
\renewcommand{\tt}{\otimes}
\newcommand{\und}[1]{\underline{#1}}
\newcommand{\z}{z}  

\newcommand{\thMW}{\theta^{\text{\rm ar}}}
\newcommand{\tph}{\widetilde{\widehat\phi_1}}
\newcommand{\Pet}{\text{\rm Pet}}





\newcommand{\thing}{ \raisebox{-6.4pt}{$\tilde{\tilde{}}$}  }   
\newcommand{\smallthing}{ \raisebox{-4.4pt}{$\scr\tilde{\tilde{}}$}  }
\newcommand{\ttilde}[1]{\overset{\smash{\thing}}{#1}}
\newcommand{\smallttilde}[1]{\overset{\smash{\smallthing}}{#1}}
\newcommand{\downhookarrow}{\hbox{$\downarrow\hskip -6.1pt\raisebox{6pt}{$\cap$}$}}


\providecommand{\bysame}{\makebox[3em]{\hrulefill}\thinspace}   
\newcommand{\hfb}{\hfill\break}
\newcommand{\margincom}[1]{\marginpar{\bf\raggedright #1}}
\newcommand{\Sec}{\S}


\numberwithin{equation}{section}
\setcounter{section}{0}
\setcounter{MaxMatrixCols}{15}


\newtheorem{theo}{Theorem}[section]
\newtheorem{lem}[theo]{Lemma}
\newtheorem{prop}[theo]{Proposition}
\newtheorem{cor}[theo]{Corollary}
\newtheorem*{atheo}{Theorem A}
\newtheorem*{btheo}{Theorem B}
\newtheorem{conj}[theo]{Conjecture}
\newtheorem{rem}[theo]{Remark}      
\newtheorem{defn}[theo]{Definition}

\newcommand{\E}{{\mathbb E}}

\newcommand{\OO}{\text{\rm O}}
\newcommand{\UU}{\text{\rm U}}

\newcommand{\OK}{O_{\smallkay}}
\newcommand{\DI}{\mathcal D^{-1}}

\newcommand{\pre}{\text{\rm pre}}

\newcommand{\Bor}{\text{\rm Bor}}
\newcommand{\Rel}{\text{\rm Rel}}
\newcommand{\rel}{\text{\rm rel}}
\newcommand{\Res}{\text{\rm Res}}
\newcommand{\TG}{\widetilde{G}}

\newcommand{\OL}{O_{\Lambda}}
\newcommand{\OLB}{O_{\Lambda,B}}

\newcommand{\p}{\varpi}

\newcommand{\cutter}{\vskip .1in\hrule\vskip .1in}

\parindent=0pt
\parskip=10pt
\baselineskip=14pt

\newcommand{\PP}{\mathcal P}
\renewcommand{\OO}{\mathcal O}
\newcommand{\BB}{\mathbb B}
\newcommand{\OBB}{O_{\BB}}
\newcommand{\Max}{\text{\rm Max}}
\newcommand{\Opt}{\text{\rm Opt}}
\newcommand{\OH}{O_H}

\newcommand{\phhat}{\widehat{\phi}}
\newcommand{\thetahat}{\widehat{\theta}}

\newcommand{\lbold}{\text{\boldmath$\l$\unboldmath}}
\newcommand{\abold}{\text{\boldmath$a$\unboldmath}}
\newcommand{\cbold}{\text{\boldmath$c$\unboldmath}}
\newcommand{\aabold}{\text{\boldmath$\a$\unboldmath}}
\newcommand{\gbold}{\text{\boldmath$g$\unboldmath}}
\newcommand{\obold}{\text{\boldmath$\o$\unboldmath}}
\newcommand{\fbold}{\text{\boldmath$f$\unboldmath}}
\newcommand{\rbold}{\text{\boldmath$r$\unboldmath}}
\newcommand{\ffbold}{\und{\fbold}}

\newcommand{\deltaBB}{\delta_{\BB}}
\newcommand{\kappaBB}{\kappa_{\BB}}
\newcommand{\aboldBB}{\abold_{\BB}}
\newcommand{\lboldBB}{\lbold_{\BB}}
\newcommand{\gboldBB}{\gbold_{\BB}}
\newcommand{\bbold}{\text{\boldmath$\b$\unboldmath}}

\newcommand{\fff}{\phi}

\newcommand{\spp}{\text{\rm sp}}

\newcommand{\pob}{\mathfrak p_{\bold o}}
\newcommand{\kob}{\mathfrak k_{\bold o}}
\newcommand{\gob}{\mathfrak g_{\bold o}}
\newcommand{\pobp}{\mathfrak p_{\bold o +}}
\newcommand{\pobm}{\mathfrak p_{\bold o -}}


\newcommand{\bb}{\frak b}

\newcommand{\bbbold}{\text{\boldmath$b$\unboldmath}}

\renewcommand{\ll}{\,\frak l}
\newcommand{\uC}{\underline{\Cal C}}
\newcommand{\uZZ}{\underline{\ZZ}}
\newcommand{\B}{\mathbb B}
\newcommand{\CL}{\text{\rm Cl}}

\newcommand{\pp}{\frak p}

\newcommand{\OKp}{O_{\smallkay,p}}

\renewcommand{\top}{\text{\rm top}}

\newcommand{\bF}{\bar{\mathbb F}_p}

\newcommand{\beq}{\begin{equation}}
\newcommand{\eeq}{\end{equation}}


\newcommand{\Dl}{\Delta(\l)}
\newcommand{\mm}{{\bold m}}

\newcommand{\FD}{\text{\rm FD}}
\newcommand{\LDS}{\text{\rm LDS}}

\newcommand{\dcM}{\dot{\Cal M}}
\newcommand{\bpm}{\begin{pmatrix}}
\newcommand{\epm}{\end{pmatrix}}

\newcommand{\GW}{\text{\rm GW}}

\newcommand{\uk}{\bold k}
\newcommand{\uo}{\text{\boldmath$\o$\unboldmath}}

\newcommand{\uz}{\und{\zeta}}
\newcommand{\duz}{\und{\dot\zeta}}
\newcommand{\ub}{\und{b}}
\newcommand{\uB}{\und{B}}
\renewcommand{\uC}{\und{C}}
\newcommand{\xp}{x_+}
\newcommand{\xm}{x_-}
\newcommand{\tu}{\tilde u}
\newcommand{\sspan}{\text{\rm span}}

\newcommand{\dbs}[1]{\frac{\d \uB_{#1}}{\d s}}
\newcommand{\dbt}[1]{\frac{\d \uB_{#1}}{\d t}}

\newcommand{\OFD}{\text{\rm OFD}}
\newcommand{\Erf}{\text{\rm Erf}}
\newcommand{\Erfc}{\text{\rm Erfc}}
\newcommand{\Arctan}{\text{\rm Arctan}}

\newcommand{\now}{\count0=\time 
\divide\count0 by 60
\count1=\count0
\multiply\count1 by 60
\count2= \time
\advance\count2 by -\count1
\the\count0:\the\count2}




\title{Theta integrals and generalized error functions}

\author{Stephen Kudla}

\maketitle

\section{Introduction}

Recently  Alexandrov, Banerjee, Manschot and Pioline, \cite{ABMP},  introduced and investigated certain generalized error functions
and used them to construct theta series for indefinite lattices of signature\footnote{Following our habit, we have taken signature $(n-2,2)$ in place of  
signature $(2,n-2)$ used in \cite{ABMP}.} $(n-2,2)$.  Their result provides an analogue of the work of Zwegers \cite{zwegers} for the case of signature $(n-1,1)$.
Roughly speaking, in each case, the choice of a suitable collection of negative (timelike) vectors allows one to cut out a certain cone of positive 
vectors and to define a convergent theta type series as the sum over lattice vectors in this cone.  Unfortunately, these nice holomorphic $q$-series are not yet modular forms 
in the variable $\tau$ in the upper halfplane, where  $q=e^{2\pi i \tau}$. 
A main result of \cite{zwegers}, for signature $(n-1,1)$, and of \cite{ABMP}, for signature $(n-2,2)$,  is that these $q$-series can be completed by the addition of a certain non-holomorphic series, obtained as a sum over all lattice vectors, 
so that the resulting function is a (non-holomorphic) modular form of weight $n/2$ in $\tau$. The completion in Zwegers involves the classical error function
$E_1(u)$, 
while the completion in \cite{ABMP} depends on the new generalized error function
$E_2(\a;u_1,u_2)$ introduced there.

On the other hand, as a very special case of the results of old joint work with John Millson, \cite{KM.I}, \cite{KM.II} and \cite{KM.IHES}, there are theta 
forms $\theta(\tau, \ph_{KM}^{(n-r,r)})$ associated to a lattice $L$ of signature $(n-r,r)$. These non-holomorphic 
series take values in closed $r$-forms on the associated symmetric space, are invariant under a finite index subgroup of the isometry group of the lattice, 
and have a modular transformation law of weight $n/2$ in $\tau$. This transformation law follows from the classical Poisson summation argument due to Siegel, \cite{siegel}, and 
hence arises in a natural way. 
The image $[\theta(\tau, \ph_{KM}^{(n-r,r)})]$ of $\theta(\tau, \ph_{KM}^{(n-r,r)})$ in cohomology inherits the modular transformation law
of the theta form and is often a holomorphic series, \cite{KM.IHES}.  . 

In the case of a lattice $L$ in an indefinite inner product space $V$ of signature $(n-1,1)$, Zwegers' theta series depends on the choice of a pair of 
negative vectors $\{C,C'\}$ lying in the same component of 
the negative cone in $V$.  The pair $\{C,C'\}$  determines a geodesic $\gamma_{C,C'}$ joining the points $C ||C||^{-1}$ 
and $C' || C'||^{-1}$ on the hyperboloid of vectors $v$ with $(v,v)=-1$ in that component. 
Then, up to sign,  Zwegers' completed theta series coincides with 
the integral of the theta form over this geodesic, \cite{kudla.zweg}, \cite{livinskyi},
$$\theta_{\text{Zwegers}}(\tau; \{C,C'\}) = -\int_{\gamma(C,C')} \theta(\tau, \ph_{KM}^{(n-1,1)}).$$
 
In the present note, we show that the non-holomorphic theta series constructed in \cite{ABMP} for lattices $L$ of signature $(n-2,2)$ can also 
be constructed as an integral of the theta form $\theta(\tau,\ph_{KM}^{(n-2,2)})$.

Let $D$ be the space of oriented 
negative $2$-planes in $V= L\tt_\Z\R$, so that  $\theta(\tau,\ph_{KM}^{(n-2,2)})$ is a closed $2$-form on $D$. The main step is to define a certain 
oriented $2$-cycle  $S$ in $D$ over which to integrate this form. 
The basic data consists of two pairs $\{C_1,C_2\}$ and $\{C_{1'},C_{2'}\}$ of negative vectors satisfying certain incidence conditions, 
(\ref{inci-1}), (\ref{inci-2}), and (\ref{inci-3}) in section 3. The relation of these conditions to those introduced in \cite{ABMP} is discussed in Remark~\ref{rem1.1} (ii).
Due to our incidence conditions, 
there are $4$ oriented negative $2$-planes, i.e., $4$ points in $D$, given by
\begin{align}\label{eq1.1}
\begin{aligned}  
z_{12}&=\sspan\{C_1,C_2\}_{\text{p.o.}}, \\ z_{1 2'}&=\sspan\{C_1,C_{2'}\}_{\text{p.o.}},
\end{aligned}
&&
\begin{aligned}
 \ z_{1'2}&=\sspan\{C_{1'},C_2\}_{\text{p.o.}},\\ z_{1'2'}&= \sspan\{C_{1'},C_{2'}\}_{\text{p.o.}},
\end{aligned}
\end{align}
lying on one component, say $D^+$ of $D$. Here the subscript `p.o.'  indicates that the given ordered pair of vectors defines the 
orientation. Any negative vector $C$ defines a copy of hyperbolic space of dimension $n-2$ in $D$ consisting of the oriented negative $2$-planes containing $C$. 
If we write $H_1$, $H_2$, $H_{1'}$ and $H_{2'}$ for the hyperbolic subspaces defined by $C_1$,  $C_2$, $C_{1'}$ and $C_{2'}$, then 
these subspace intersect in precisely $8$ points, $4$ on each component
of $D$. The intersection points on $D^+$ are precisely the points (\ref{eq1.1}), for example, 
$z_{12}=H_1\cap H_2\cap D^+$, etc. 
The points $z_{12}$ and $z_{12'}$ lie on $H_1$ and are joined by a unique geodesic $\gamma_1$ in $H_1$. This and the analogously defined geodesics $\gamma_{1'}$, $\gamma_2$ 
and $\gamma_{2'}$ form a quadrilateral loop in $D$.  Let $S = S(C_1,C_2, C_{1'},C_{2'})$ be the geodesic surface in $D$ filling in this quadrilateral. 

Recall that the theta form is defined as follows. 
For $x\in V$, let
$$\ph_{KM}^o(x) =2 (\ \o_1(x)\wedge \o_2(x) - \frac{1}{4\pi} \O\,)\,e^{-2\pi R(x,z)},$$
and, for $\tau=u+iv$ in the upper half-plane, let 
$$\ph_{KM}(\tau,x) = q^{\frac12(x,x)}\,\ph_{KM}^o(x \sqrt{v}), \qquad q= e(\tau) = e^{2\pi i \tau},$$
where we write $\ph_{KM}$ in place of $ \ph_{KM}^{(n-2,2)}$, since the signature is fixed from now on. 
Here $R(x,z) = -(\pr_z(x),\pr_z(x))$, where $\pr_z(x)$ is the orthogonal projection of $x$ to the negative $2$-plane $z$, $\o_1(x)$ and $\o_2(x)$ 
are $1$-forms on $D$ determined by $x$, and $\O$ is an invariant $2$ form on $D$, independent of $x$. 
Let $L^\vee\supset L$ be the dual lattice of $L$.  For $\mu\in L^\vee/L$, the theta form is the closed $2$-form given by
\beq\label{theta.form}
\theta_\mu(\tau; \ph_{KM}) = \sum_{x\in L+\mu}\, \ph_{KM}(\tau,x).
\eeq

Since $S$ is compact, we can compute the integral of $\theta_\mu(\tau; \ph_{KM})$ termwise and so the essential 
result is the evaluation of the integral
\beq\label{basic-integral}
I(x;S) := \int_S \ph^o_{KM}(x).
\eeq

\begin{atheo}  Let $S= S(C_1,C_2, C_{1'},C_{2'})$ be the oriented surface determined by the negative pairs $\{C_1,C_2\}$ and $\{C_{1'},C_{2'}\}$
satisfying the incidence relations (\ref{inci-1}), (\ref{inci-2}), and (\ref{inci-3}).\hfb
(a) For any $x\in V$,  
$$I(x/\sqrt{2};S) =  - \frac14\,\big(\ E_2(C_1,C_2;x) - E_2(C_1,C_{2'};x) -E_2(C_{1'},C_2;x) + E_2(C_{1'},C_{2'};x)\ \big),$$
where $E_2(C,C';x)$ is the `boosted' generalized error function defined in (3.38) of \cite{ABMP}.\hfb
(b) In particular, 
\begin{align*}
I(0;S) &= -\Arctan\left( \frac{(C_1,C_2)}{\sqrt{\Delta_{12}}} \right)+\Arctan\left( \frac{(C_1,C_{2'})}{\sqrt{\Delta_{12'}}} \right) \\
\nass
{}&\qquad\qquad
+  \Arctan\left( \frac{(C_{1'},C_2)}{\sqrt{\Delta_{1'2}}} \right) - \Arctan\left( \frac{(C_{1'},C_{2'})}{\sqrt{\Delta_{1'2'}}} \right).
\end{align*}
Here
$$\Delta(C,C') = (C,C)(C',C') - (C, C')^2,$$
and $\Delta_{12} = \Delta(C_1,C_2)$,  etc.
\end{atheo}

This result is first proved in the `generic' case, where $(x,C)\ne 0$ for all $C\in \{C_1,C_2,C_{1'},C_{2'}\}$ by an 
application of Stoke's theorem. 
Recall that for any nonzero $x\in V$, there is a submanifold
$$D_x=\{\,z\in D\mid x\in z^\perp\,\}$$
 of $D$, 
which is the zero locus of the function $R(x,\cdot)$ and is empty unless $(x,x)>0$.  
If $(x,x)>0$, then $D_x$ has codimension $2$. Indeed, in the hermitian model, it is either empty or a complex divisor on $D$. 
On the set 
$D-D_x$, we define a $1$-form $\psi(x)$ with
$$d\psi(x) = \ph_{KM}^o(x).$$
As in \cite{ABMP}, let 
\beq\label{Phi2x}
\P_2(x) = \frac14\,[\ \sgn(x,C_1) - \sgn (x,C_{1'})\ ][\ \sgn(x,C_2)- \sgn(x,C_{2'})\ ].
\eeq
where we take $\sgn(0)=0$. 
Then, for $x$ generic,  $D_x\cap S$ is non-empty if and only if $\P_2(x)\ne0$!
To apply Stokes' theorem, we must cut out an $\e$-disk around $D_x\cap S$. It turns out that 
$$\int_{\d S}\psi(x)$$
is easy to evaluate and leads immediately to the function
$$f(a,b):=-\frac{1}{2\pi} \, b\, e^{-2\pi b^2}\,\int_0^a  \frac{e^{-2\pi t^2}}{b^2 + t^2}\,dt.$$ 
But this function is given, in turn, as 
$$f(a,b) = -\frac14\,\tilde e_2(a\sqrt{2},b\sqrt{2}),$$
where $\tilde e_2$ is one of the building blocks in \cite{ABMP} in the generalized error function. 
Now using the identity (\ref{Etoe2}) relating $\tilde e_2$ and $E_2$, we obtain the expression in (i) of Theorem~A together with an extra term $\P_2(x)$, 
cf. Corollary~\ref{cor4.4}. 
This extra term is cancelled by the contribution of the excised $\e$-disk!  The case of non-generic $x$ is obtained by continuity, where we note that 
the singularities arising from $\psi(x)$ were not present in the original integral, which depends smoothly on $x$ and $S$. 

Our main global result is then the following. 

\begin{btheo} Let $S=S(C_1,C_2,C_{1'},C_{2'})$ be as in Theorem~A.\hfb 
(a) The series
\beq\label{theta-integral}
I_\mu(\tau;S) := \int_S\theta_\mu(\tau; \ph_{KM}) = \sum_{x\in L+\mu} I(\sqrt{v}\,x;S)\,q^{\frac12(x,x)}
\eeq
is a non-holomorphic modular form of weight $n/2$ with the same transformation law as $\theta_\mu(\tau,\ph_{KM})$. \hfb
(b) For $\mu\in L^\vee/L$,  the $q$-series
\beq\label{hol.part}
\sum_{x\in L+\mu}  \P_2(x) \, q^{\frac12(x,x)}
\eeq
is termwise absolutely convergent. \hfb
(c)
The modular form $I_\mu(\tau;S)$ is its modular completion\footnote{We assume that $L$ is even integral, so that the 
characteristic vector $p$ used in (2.1) of \cite{ABMP} can be taken to be $0$ and hence does not appear in our theta series.
This is only done to simplify the notation;  our results hold in the general case. },
as defined in \cite{ABMP}.
In particular, 
\beq\label{Main-id}
I_\mu(\tau;S) = -\theta_{\text{ABMP}}(\tau;C_1,C_2,C_{1'},C_{2'};\mu),
\eeq
where $\theta_{\text{ABMP}}(\tau;C_1,C_2,C_{1'},C_{2'};\mu)$ is the series defined in \cite{ABMP}. 
\end{btheo}

\begin{rem} \label{rem1.1} (i)  Since the theta form is closed, the integral on the right side of 
(\ref{theta-integral}) does not depend on the choice of the oriented surface $S$ with boundary $ \gamma_1 + \gamma_{2'} - \gamma_{1'} - \gamma_2.$\hfb
(ii)  Our incidence conditions (\ref{inci-1}), (\ref{inci-2}), and (\ref{inci-3}) are a subset of those imposed in \cite{ABMP}, (4.2) and (4.6), 
and they are the natural conditions giving rise to the $4$ points
$\{z_{12}, z_{1'2},z_{12'},z_{1'2'}\}$ on $D^+$ and  the boundary frame of geodesics connecting them.   
It is a pleasant bonus that our conditions are sufficient to ensure the convergence of the $q$-series (\ref{hol.part}).
Their relation to the set of conditions \cite{ABMP} (4.2) alone is not clear. \hfb
(iii) Of course, part (c) is just a restatement of the result of \cite{ABMP} relating the expression $\P_2(x)$ and the quantity $I(x;S)$. \hfb
(iv)
Note that, in passing from Theorem~A to the theta series in Theorem~B, the at first sight anomalous $\sqrt{2}$
is precisely what is introduced by the $\sqrt{2 \tau_2}$ in the argument of $\P$ in the definition of the theta series 
in (2.1) of \cite{ABMP}.
\end{rem}

Our construction also provides the following geometric interpretation of the quantity 
$\P_2(x)$.
\begin{prop}\label{prop.inter}
Suppose that $x\in V$ satisfies the regularity condition in Theorem~A and that $D_x\cap S$ is non-empty. Then $D_x\cap S$ consists of a single point and 
$$\P_2(x) = I(S,D_x)$$
where $I(S,D_x)$ is the(local) intersection number of the (oriented) cycles $S$ and $D_x$ in $D$. 
\end{prop}
\begin{cor}
The holomorphic  part (\ref{hol.part}) of the indefinite theta function \hfb
$I_\mu(\tau;S)$ 
can be written as
$$\sum_{x\in L+\mu} I(S,D_x)\,q^{\frac12(x,x)}.$$
\end{cor}

As is evident from this description, our result technically depends only on the special functions results of section 3 of \cite{ABMP}, i.e., 
in our approach, the completed theta series is constructed directly as an integral of the theta form and the results about generalized error functions, 
in particular (\ref{Etoe2}), etc., then identify it as a completion 
of the $q$-series (\ref{hol.part}).   But of course, the whole business depends entirely on 
the original ideas of \cite{ABMP} about constructing a signature $(n-2,2)$ analogue of Zwegers' 
theta series,  the ingenious choice of the correct input data $\{\{C_1,C_2\},\{C_{1'},C_{2'}\}\}$, the function $\P_2(x)$, etc. 
Aside from the connection with the theta form, our main addition is the introduction of the surface $S$ over which 
we take our integral and the interpretation of 
$\P_2(x)$ as an intersection number.  Hopefully, these ideas will prove useful in the construction of examples 
for more general signatures proposed in Section 6 of \cite{ABMP}.  The theta forms $\theta(\tau,\ph_{KM}^{(n-r,r)})$ 
are available in this situation and can undoubtedly be used in the same way, cf. the discussion in section~\ref{section-other-sigs}.

Here is a brief description of the contents of the paper.  In section 2, we review the construction of the $2$-form $\ph_{KM}^o(x)$
and show that, for $x\ne 0$ and away from the set $D_x$, there is an explicitly constructed $1$-form $\psi(x)$ with 
$d\psi(x) = \ph_{KM}^o(x)$.  
In section 3, we describe how the data $\{C_1,C_2\}$,  $\{C_{1'},C_{2'}\}$ satisfying the incidence relations of \cite{ABMP}
gives rise the a quadrilateral of geodesics and a spanning surface $S$. In section 4, we compute the basic integral of $\ph_{KM}^o(x)$ 
over the surface $S$. This is done using Stoke's theorem and the $1$-form $\psi(x)$, where the singularity along $D_x$ 
necessitates a limit procedure when $D_x\cap S$ is non-empty. Also in this section, we prove Proposition~\ref{prop.inter}. 
In section~\ref{section-irregular}, we take care of the cases where the regularity condition fails. In section~\ref{section-shadows}, we discuss how 
the computation of the shadows i.e., the image of $I_\mu(\tau;S)$ under the lowering operator, can be done in our language, and in section~\ref{section-other-sigs}
we sketch how things should go for other signatures. Here the surface $S$ is generalized to a hypercube, and the convergence of the 
associated $q$-series is proved. 
Convenient coordinates for our computation are explained in an Appendix, section 6.

We thank Luis Garcia for a useful conversation. 

\subsection{Notation}\label{notation.section}

Following the conventions of \cite{ABMP}, we write
$$\Delta(C,C') = (C,C)(C',C') - (C, C')^2  = \det\bpm (C,C)& (C,C')\\ (C',C)&(C',C')\epm,$$
$\Delta_{12} = \Delta(C_1,C_2)$,  etc.  More generally, for any subset $I$ of $\{1,2,1',2'\}$, $\Delta_{I}$ is the determinant of the matrix 
of inner products $((C_i,C_j))_{i,j\in I}$. In addition, write
$$C_{1\perp 2} = C_1 - \frac{(C_1,C_2)}{(C_2,C_2)}\,C_2,\qquad\text{etc.}$$
For a vector $C$ with $(C,C)\ne 0$, we write
$$||C|| = |(C,C)|^{\frac12}, \qquad \und{C} = \frac{C}{||C||}.$$
We use the convention $\sgn(0)=0$ to extend the sign function $\sgn$ from $\R^\times$ to $\R$. 

For collections $x = [x_1,\dots, x_r]\in V^r$ and $y=[y_1,\dots,y_s]\in V^s$, we let 
$$(x,y) = ((x_i,y_j))\in M_{r,s}(\R).$$
Note that for $a\in \GL_r(\R)$ and $b\in \GL_s(\R)$, 
$$(xa,yb) = {}^ta (x,y)b.$$

\section{The basic setup}
Let $V$, $(\,,\,)$  be an indefinite inner product space of signature $(n-2,2)$ and let $D$ be the space of 
oriented negative $2$-planes in $V$. Let 
$$\OFD = \{ \, \zeta = [\zeta_1,\zeta_2] \in V^2\mid (\zeta, \zeta)=-1_2\ \}$$
be the space of oriented orthonormal frames and let $\pi:\OFD \rightarrow D$ be the projection sending an oriented 
frame to the oriented $2$-plane it spans. 

For $z\in D$ with properly oriented orthonormal basis $\zeta = [\zeta_1,\zeta_2]$, the orthogonal projection 
of $x\in V$ to $z$ is given by 
$$\pr_z(x) = \zeta (\zeta,\zeta)^{-1}(\zeta,x).$$
Then
$$R(x,z) = -(\pr_z(x),\pr_z(x)) =(x,\zeta)(\zeta,x),$$
and the majorant is
$$(x,x)_z = (x,x) +2 R(x,z).$$
Note that $z\in D_x$, i.e., $z$ is in $x^\perp$ if and only if $R(x,z)=0$.

\subsection{The Schwartz form}

We recall the definition and basic properties of the Schwartz form constructed in \cite{KM.I}, \cite{KM.II}. 
Our conventions about differential forms are explained in more detail in the Appendix~\ref{subsec-tan-vecs}.

Define a Schwartz function on $V$ valued in $2$-forms on $D$ as follows.  
For $\eta$ and $\mu \in T_z(D) = \Hom(z,z^\perp)$, 
we choose $\zeta\in \OFD$ with $\pi(\zeta)=z$ and write
$$ [\eta_1,\eta_2] =[\eta(\zeta_1), \eta(\zeta_2)], \qquad [\mu_1,\mu_2] = [\mu(\zeta_1), \mu(\zeta_2)] \ \in U(z)^2.$$
Here, for convenience, we write $z^\perp=U(z)$. 
Then let 
$$\o_1(x)\wedge\o_2(x)(\eta,\mu) = (x,\eta_1)(x,\mu_2) - (x,\eta_2)(x,\mu_1) = \det\bpm (x,\eta_1)&(x,\eta_2)\\ (x,\mu_1)&(x,\mu_2)\epm,$$
and 
$$\O(\eta,\mu) = (\eta_1,\mu_2) - (\eta_2,\mu_1).$$
These expressions are independent of the choice of $\zeta\in \OFD$ and define $2$-forms on $D$. The basic Schwartz form\footnote{Note that we have included an extra factor of $2$ 
here compared with the expression given in Proposition~5.1 of \cite{KM.I}. This is done to compensate for the factor $2^{-q/2}$ that occurs in the fiber integral in 
Proposition~6.2 of \cite{KM.II} and hence to correctly normalize the Poincar\'e dual form.}
$$\ph_{KM}(x,z)= 2\,\big(\ \o_1(x)\wedge \o_2(x) - \frac{1}{4\pi}\O\ \big)\,\exp(-\pi  (x,x)_z),$$
where $(x,x)_z$ is the majorant,  is a closed $2$-form on $D$.  This is a very special case of the forms defined in \cite{KM.I} and \cite{KM.II}. 
Note that 
$$\ph_{KM}(0) = -\frac{1}{2\pi}\,\O.$$
For convenience, since $x$ will often be fixed in our discussion, we write 
$$\ph_{KM}(x) = e^{-\pi(x,x)}\,\ph_{KM}^o(x).$$
Note that the pullback of $\ph_{KM}^o(x)$ under $\pi:\OFD \rightarrow D$ is
$$\pi^*(\ph_{KM}^o(x)) = 2\,\big[\,(x,d\zeta_1)\wedge (x,d\zeta_2) - \frac1{4\pi}\,(d\zeta_1,d\zeta_2)\,\big] \,e^{-2\pi R}.$$

On the set 
$$D - \{[x,z]\in V\times D\mid z\in D_x\}, $$
define a $1$-form $\psi(x)$ whose pullback to $\OFD$ is given by
\begin{align}\label{def-psi}
\pi^*(\psi(x))  &= -\frac{1}{2\pi} e^{-2\pi R}\,\big(\ R^{-1}(\,(x,\zeta_1)(x,d\zeta_2)-(x,\zeta_2)(x,d\zeta_1) \,) - (\zeta_2,d\zeta_1)\ \big).
\end{align}

The proof of the following proposition is given in Appendix~\ref{subsec-proof-dpsi}.
 \begin{prop} \label{dpsi}
  On the set $D - D_x$, 
 $$\ph_{KM}(x) = d \psi(x)\,e^{-\pi(x,x)}.$$
 \end{prop} 

\section{The spanning surface of a frame}

In this section, we explain how to obtain an oriented $2$-cycle in $D$ from 
the basic data introduced in \cite{ABMP} consisting of two pairs of 
negative\footnote{Note that our quadratic form is the negative of theirs, so for us, the term `timelike' refers to vectors $C$ with 
$ Q(C) = (C,C)<0$.} vectors $\{C_1,C_2\}$  and $\{C_{1'}, C_{2'}\}$ 
subject to a certain collection of incidence relations.  We impose these relations step by step, observing how the 
geometry we want emerges. 

With the notation of \cite{ABMP},  as explained in section~\ref{notation.section},
we impose the incidence relations\footnote{Notice that replacing signature $(2,n-2)$ with signature $(n-2,2)$ amounts to  changing the sign of the quadratic form. 
The quantities $\Delta(C,C')$ are invariant under this change, but the quantities $(C,C')$ change sign. This gives the equivalence of 
our conditions (\ref{inci-1}) and (\ref{inci-2}) with the conditions (4.6) in \cite{ABMP}.}
\begin{align}
&\Delta_{12}, \ \Delta_{1'2}, \ \Delta_{12'},\ \Delta_{1'2'}\ >\ 0,\label{inci-1}\\
\noalign{\vskip -10pt and}
& (C_{2\perp1},C_{2'\perp1}), \  (C_{2\perp1'},C_{2'\perp1'}), \ (C_{1\perp2},C_{1'\perp2}),\ (C_{1\perp2'},C_{2\perp2'}) \ <\ 0.\label{inci-2}
\end{align} 

Condition (\ref{inci-1}) is equivalent to the requirement at the pairs $\{C_1,C_2\}$, $\{C_{1'},C_2\}$, $\{C_1,C_{2'}\}$
and $\{C_{1'},C_{2'}\}$ span negative $2$-planes 
\beq\label{Ps} P_{12},\  P_{1'2},\ P_{12'},\ P_{1'2'}, 
\eeq
where, for the moment we ignore the orientations.

For any negative vector $C\in V$, we let 
$$H_C = \{ z\in D\mid C\in z\},$$
and $V_C = C^\perp$. 
Then $V_C$ has signature $(n-2,1)$ and we have an isomorphism
\beq\label{hyperb}
\{ v\in V_C \mid (v,v)=-1\} \isoarrow H_C, \qquad v \mapsto \sspan\{C,v\}_{\text{p.o.}}.
\eeq
Here the subscript `p.o.' indicates that the given ordered frame is properly oriented. 
In particular, $H_C$ is a copy of hyperbolic $n$-space in $D$. Note that $H_C$ has two components and that vectors $v$ and $v'$ 
in $V_C$ with $(v,v)=-1=(v',v')$ lie in the same component precisely when $(v,v')<0$.  

From the given set of negative vectors $\{C_1,C_2,C_{1'},C_{2'}\}$ we obtain hyperbolic subspaces $H_1$, $H_{1'}$, $H_2$ and $H_{2'}$ in $D$, where 
we write $H_1 = H_{C_1}$, $V_1=V_{C_1}$, etc. Assuming condition (\ref{inci-1}), the intersections $H_1\cap H_2$, $H_1\cap H_{2'}$, $H_{1'}\cap H_2$,
and $H_{1'}\cap H_{2'}$ of these  
hyperbolic subspaces are easily determined. For example, 
$$H_1\cap H_2 = \{\,\sspan\{C_1,C_2\}_{\text{p.o.}} , \sspan\{C_1,-C_2\}_{\text{p.o.}}\, \}.$$
Since  $\sspan\{C_1,C_2\} = \sspan\{C_1,C_{2\perp1}\}$ is a negative $2$-plane, it follows that
$C_{2\perp1}$ is a negative vector in $V_1$ as is $C_{2'\perp1}$, by the same argument. Since $V_1$ has signature $(n-2,1)$ 
the inner product\footnote{This is why we have replaced the $\ge0$ in (4.6b)--(4.6e) with $<0$ in (\ref{inci-2})}
 $(C_{2\perp1},C_{2'\perp1})\ne 0$, and condition (\ref{inci-2}) implies that the vectors $C_{2\perp1}$ and $C_{2'\perp1}$ lie in the same component
 of the negative cone in $V_1$.  The same discussion applies to the other pairs of projections, $C_{1\perp2}$, $C_{1'\perp2}$, etc. 
 
 Next we add the condition 
 \beq\label{inci-3}
 \Delta_{11'22'}>0.
 \eeq
 This is (4.2b) in \cite{ABMP}.
 This implies that $U=\sspan\{C_1,C_2,C_{1'},C_{2'}\}$ is a $4$ dimensional subspace of $V$ of signature $(2,2)$. It follows that the negative 
 $2$-planes in (\ref{Ps}) are distinct.  
 
 Thus, under conditions (\ref{inci-1}) , (\ref{inci-2}), and (\ref{inci-3}), we obtain 
  $8$ points in $D$, the $4$ listed in (\ref{eq1.1}) and the $4$ obtained from them by reversing the orientations.

By condition (\ref{inci-2}), $C_{2\perp1}$ and $C_{2'\perp1}$ lie in the same component of the negative cone in $V_1$, and hence
there is a unique geodesic in $H_1$ joining the points $z_{12}$ and $z_{12'}$ given by\footnote{We make no claim that $t$ is the arc-length parameter.} 
\beq\label{eq-gamma1}
\gamma_1(t) = \sspan\{ C_1, (1-t)C_{2} + t C_{2'}\}_{\text{p.o.}} = \sspan\{ \uC_1, \uB_{2\perp1}(t)\}_{\text{p.o.}}, \qquad t\in [0,1],
\eeq
where 
$$B_{2\perp 1}(t) = (1-t)C_{2\perp1} + t C_{2'\perp1}, \qquad \uB_{2\perp1}(t) = \frac{B_{2\perp1}(t)}{||B_{2\perp1}(t)||}, \quad \uC_1 = \frac{C_1}{||C_1||}.$$  
Note that the second expression gives an orthonormal basis of $\gamma(t)$ whose second component lies in $V_1$. In particular, 
$\gamma_1$ is the image in $D$ of the geodesic in (\ref{hyperb}) joining the points $\und{C}_{2\perp1}$ and $\und{C}_{2'\perp1}$. 

There are analogous geodesics $\gamma_{1'}$ from $z_{1'2}$ to $z_{1'2'}$, $\gamma_{2}$ from $z_{12}$ to $z_{1'2}$ 
and $\gamma_{2'}$ from $z_{12'}$ to $z_{1'2'}$.  Explicitly, these are given by 
\begin{align*}
\gamma_{1'}(t)&=\sspan\{ \uC_{1'}, \uB_{2\perp1'}(t)\}_{\text{p.o.}}, \qquad t\in [0,1],\\
\nass
\gamma_2(s)&= \sspan\{\uB_{1\perp2}(s), \uC_{2}\}_{\text{p.o.}}, \qquad\  s\in [0,1],\\
\nass
\gamma_{2'}(s) &=\sspan\{\uB_{1\perp2'}(s), \uC_{2'}\}_{\text{p.o.}}, \qquad s\in [0,1],
\end{align*}
where
\begin{align*}
B_{2\perp1'}(t) & = (1-t)C_{2\perp 1'} + t C_{2'\perp1'},\\
\nass
B_{1\perp2}(s) &= (1-s) C_{1\perp2}+s C_{1'\perp2}\\
\nass
B_{1\perp2'}(s) &= (1-s) C_{1\perp2'} + s C_{1'\perp 2'}.
\end{align*}
Altogether, they form a  closed loop in $D$. 
Note that the first component is fixed for $\gamma_1$ and $\gamma_{1'}$ and the second component is fixed for 
$\gamma_2$ and $\gamma_{2'}$. This will result in a sign change later in the calculation.

We parametrize a $2$-cycle $S= S(C_1, C_2, C_{1'}, C_{2'})$ in $D$ by 
\beq\label{param.S}
\phi: [0,1]^2 \lra D, \qquad \phi(s,t) = \sspan\{ B_1(s), B_2(t)\}_{\text{p.o.}}
\eeq
where
\beq\label{param.S.lift}
B_1(s) = (1-s) C_1 + s C_{1'}, \qquad B_2(t) = (1-t)C_2 + t C_{2'}.
\eeq
This oriented $2$-cycle fills in the frame and has oriented boundary
$$\d S = \gamma_1 + \gamma_{2'} - \gamma_{1'} - \gamma_2.$$
Here it will sometimes be useful to write
\beq\label{def-gammaC}
\gamma = \gamma_{\{C_1,C_2,C_{1'},C_{2'}\}} = \gamma_1 + \gamma_{2'} - \gamma_{1'} - \gamma_2,
\eeq
to emphasize the dependence on the given data.

\begin{rem}  Note that our incidence conditions (\ref{inci-1}), (\ref{inci-2}), and (\ref{inci-2}) are a subset of the incidence 
relations (4.2) and (4.6) imposed in \cite{ABMP}.  We did not check whether their additional conditions are consequences of the ones we impose.
\end{rem}

\section{The cycle integral}

Our main goal is to compute the integral 
$$ \int_S \ph_{KM}^o(x).$$

\subsection{Some geometry}

The first step is to determine the intersection of $S$ with the singular set $D_x$ of $\psi(x)$. 
For the moment, we consider only the generic case. 

\begin{defn}
A vector $x$ is {\bf regular} with respect to $\{C_1,C_{1'},C_2,C_{2'}\}$  if 
$(C,x)\ne 0$
for all $C\in \{C_1,C_{1'},C_2, C_{2'}\}$. 
\end{defn}

\begin{lem}\label{int-lemma} 
 (i) Suppose that $x$ is regular with respect to $\{C_1,C_{1'},C_2,C_{2'}\}$. \hfb
Then the set $D_x\cap S$ is non-empty if and only if $\P_2(x)\ne 0$, and, in this case, 
$$D_x\cap S = \sspan\{B_1(s_0), B_2(t_0)\},$$
where
\beq\label{null-point}   
s_0 = \frac{(x,C_1)}{(x,C_1)-(x,C_{1'})}, \qquad t_0= \frac{(x,C_2)}{(x,C_2)-(x,C_{2'})}.
\eeq
(ii) Suppose that $x\in V$ is any vector with $\P_2(x)\ne0$. Then, $D_x\cap S$ consists of a single point given by (\ref{null-point}).
\end{lem}
\begin{proof}
A point  $\phi(s,t)=\sspan\{B_1(s), B_2(t)\}$ lies in $D_x\cap S$ precisely when 
\beq\label{null-point-eqs}
(1-s)(x,C_1)+ s (x, C_{1'})=0=(1-t)(x,C_2)+t (x,C_{2'}), \qquad s, t \in [0,1].
\eeq
Under the regularity assumption, if $(x,C_1)\ne (x,C_{1'})$ and $(x,C_2)\ne (x,C_{2'})$,  the given pair $s_0$, $t_0$ is the unique solution. 
These lie in $[0,1]$ precisely when $\P_2(x)\ne 0$. Indeed, under the regularity assumption, they both lie in $(0,1)$.  
If either $(x,C_1)= (x,C_{1'})$ or $(x,C_2)= (x,C_{2'})$, then there is no solution. 
\end{proof}

If $\P_2(x)\ne 0$ and for $\e>0$ sufficiently small so that $s_0\pm \e$ and $t_0\pm\e$ lie in $(0,1)$, let
\begin{align*}
C_1^\e(x) &= B_1(s_0-\e)\qquad C_{1'}^\e(x) = B_1(s_0+\e)\\
\nass
C_2^\e(x)&= B_2(t_0-\e)\qquad C_{2'}^\e(x) = B_2(t_0+\e).
\end{align*}
The collection of negative vectors $\{C_1^\e(x), C_2^\e(x), C_{1'}^\e(x), C_{2'}^\e(x)\}$ satisfy the same `incidence' conditions (\ref{inci-1}), (\ref{inci-2}), and (\ref{inci-2}) as the 
collection $\{C_1,C_2, C_{1'}, C_{2'}\}$ and determine a surface
$$S^\e(x):=S(C_1^\e(x), C_2^\e(x), C_{1'}^\e(x), C_{2'}^\e(x)) \ \subset\ S(C_1, C_2, C_{1'}, C_{2'}) = S,$$
containing the point $D_x\cap S = \phi(s_0,t_0)$. 
We write the boundary of $S^\e(x)$ as
$$\d\,S^\e(x) =  \gamma_1^\e(x) + \gamma_{2'}^\e(x) - \gamma_{1'}^\e(x) - \gamma_2^\e(x),$$
for curves defined by the analogues of (\ref{eq-gamma1}), etc.
Thus 
$$\d\, S^\e(x) = \gamma_{\{C_1^\e(x), C_2^\e(x), C_{1'}^\e(x), C_{2'}^\e(x)\}}= : \gamma^\e(x),$$
in the notation introduced in (\ref{def-gammaC}). 

Let 
$$
S'_\e = \begin{cases} S - \text{int}\,S^\e(x) &\text{if $\P_2(x)\ne0$}\\
\nass
S&\text{if $\P_2(x)=0$.}
\end{cases}
$$
so that 
$$
\d S'_\e = \begin{cases} \d S - \gamma^\e(x) &\text{if $\P_2(x)\ne0$}\\
\nass
\d S&\text{if $\P_2(x)=0$.}
\end{cases}
$$
Then by Stoke's theorem we have
$$
\int_{S'_\e} \ph_{KM}^o(x) = \int_{\d S'_\e} \psi(x).
$$

\subsection{The contribution of $\d S$}
We compute the integral of the $1$-form $\psi(x)$ around the boundary $\gamma = \d S$. 

First consider the integral over $\gamma_1$. 
The second expression in (\ref{eq-gamma1}) gives a lift of $\gamma_1$ to a curve in $\OFD$ with 
tangent vector 
$$\eta(t) = [\eta_1(t),\eta_2(t)]=  [\dot{\und{C}}_1, \dot{\und{B}}_{2\perp1}(t)] = [0, \dot{\und{B}}_{2\perp1}(t)].$$
Note that 
$$(\und{C}_1, \dot{\und{B}}_{2\perp1}(t))=0, \quad\text{ and } ( \und{B}_{2\perp1}(t), \dot{\und{B}}_{2\perp1}(t))=0,$$ 
since $( \und{B}_{2\perp1}(t), \und{B}_{2\perp1}(t))=-1$. 
Thus the term in $\psi(x)$ involving $(\zeta_2,d\zeta_1)$ vanishes along $\gamma_1$,  and we have simply
\begin{align}\label{gamma1-integral}
\int_{\gamma_1}\psi(x) & =  -\frac{1}{2\pi}\int_0^1  \frac{e^{-2\pi R}}{R}\,(x,\und{C}_1)(x, \dot{\und{B}}_{2\perp1}(t))\, dt,
\end{align}
where
$$R = (x,\und{C}_1)^2 + (x, \und{B}_{2\perp1}(t))^2.$$
Let $u = (x, \und{B}_{2\perp1}(t))$ so that $du = (x, \dot{\und{B}}_{2\perp1}(t))\, dt$, and we have
\begin{align*}
\int_{\gamma_1}\psi(x) & =  -\frac{1}{2\pi} \, (x,\und{C}_1)\, e^{-2\pi (x, \und{C}_1)^2}\,\int_\a^{\a'}  \frac{e^{-2\pi u^2}}{(x,\und{C}_1)^2 + u^2}\,du.
\end{align*}
where $\a = (x,\und{C}_{2\perp1})\quad\text{and}\quad \a'= (x,\und{C}_{2'\perp1})$. 

Now, recalling the function 
\beq\label{e2-int}
\tilde{e}_2(a,b) =  \frac{2}{\pi}\, b\,e^{-\pi b^2}\,\int_0^{a} e^{-\pi v^2}\,(b^2+v^2)^{-1}\,dv,
\eeq
defined in \cite{ABMP}, (3.25), we can write
\begin{align}\label{gamma1-integral}
4\,\int_{\gamma_1}\psi(x) & =  \tilde e_2\big((x,\und{C}_{2\perp1})\sqrt{2},(x,\und{C}_1)\sqrt{2}\big) -\tilde e_2\big((x,\und{C}_{2'\perp1})\sqrt{2},(x,\und{C}_1)\sqrt{2}\big).
\end{align}

The values of the other integrals are obtained by permutation of indices. Note that there will be an additional change in sign 
in the $\gamma_2$ and $\gamma_{2'}$ integrals due to the fact noted above that the `fixed component' is then the second component.

\begin{prop}\label{Nice1} Under the regularity assumption, 
\begin{align*}
4\int_{\d S} \psi(x/\sqrt{2}) ={} &\tilde e_2\big((x,\und{C}_{2\perp1}),(x,\und{C}_1)\big)\ -\tilde e_2\big((x,\und{C}_{2'\perp1}),(x,\und{C}_1)\big)\\
\nass
-{}\ &\tilde e_2\big((x,\und{C}_{1\perp2'}),(x,\und{C}_{2'})\big) +\tilde e_2\big((x,\und{C}_{1'\perp2'}),(x,\und{C}_{2'})\big)\\
\nass
-{}\ &\tilde e_2\big((x,\und{C}_{2\perp1'}),(x,\und{C}_{1'})\big) +\tilde e_2\big((x,\und{C}_{2'\perp1'}),(x,\und{C}_{1'})\big)\\
\nass
+{}\ &\tilde e_2\big((x,\und{C}_{1\perp2}),(x,\und{C}_2)\big)\ -\tilde e_2\big((x,\und{C}_{1'\perp2}),(x,\und{C}_2)\big).
\end{align*}
\end{prop} 

In order to compare this quantity with what occurs in \cite{ABMP}, we note the identity which follows from equation (3.24) there
once the `boosting' is taken into account:
\begin{align}\label{Etoe2}
E_2(C_1,C_2;x) = -\tilde e_2((x,\und{C}_{1\perp2}), &(x,\und{C}_2)) -\tilde e_2((x,\und{C}_{2\perp1}), (x,\und{C}_1))\\
\nass
{}&\qquad\qquad\qquad + \sgn((x,\und{C}_{2}))\,\sgn((x,\und{C}_1)).\notag
\end{align}
Using this, we obtain a nice expression for our integral. 
\begin{cor}\label{cor3.3}
\begin{align*}
\int_{\d S} \psi(x/\sqrt{2}) &= \P_2(x)-\frac14\,\big(\, E_2(C_1,C_2;x) - E_2(C_1,C_{2'};x) \\
\nass
{}&\qquad\qquad\qquad\qquad -E_2(C_{1'},C_2;x) + E_2(C_{1'},C_{2'};x)\, \big).
\end{align*}
\end{cor}

\subsection{The contribution of the singular point} 

The integral of $\psi(x)$ around $\d S^\e(x)$ is given by the expression in Corollary~\ref{cor3.3} 
with the collection $\{C_1,C_2, C_{1'}, C_{2'}\}$ replaced by $\{C_1^\e(x), C_2^\e(x), C_{1'}^\e(x), C_{2'}^\e(x)\}$. 
Thus we have 
\begin{align*}
\int_{\d S^\e(x)} \psi(x/\sqrt{2}) &= \P_2^\e(x)-\frac14\,\big(\, E_2(C_1^\e,C_2^\e;x) - E_2(C_1^\e,C_{2'}^\e;x) -E_2(C_{1'}^\e,C_2^\e;x) + E_2(C_{1'}^\e,C_{2'}^\e;x)\,\big),
\end{align*}
where we have written $C_1^\e$ in place of $C_1^\e(x)$ etc., to lighten the notation. 
Note that at $\e=0$ we have $C_1^0 = C_{1'}^0= B_1(s_0)$ and $C_2^0 = C_{2'}^0=B_2(t_0)$. Thus the sum of the $E_2$ terms vanishes in the limit as $\e$ goes to $0$. 
Also, as $\e$ runs over $[s_0,0)$, the vector $C_1^\e(x)$ runs from 
$C_1$ to $C_1^0(x)$ and the quantity $(x,C_1^\e(x))$ does not vanish and, in particular, does not change sign along this path. Similarly for the other corners. 
Thus $\lim_{\e\rightarrow 0} \P_2^\e(x) = \P_2(x)$, and we have the following.
\begin{cor}\label{cor3.4}
$$\lim_{\e\rightarrow 0} \int_{\d S^\e(x)} \psi(x/\sqrt{2}) = \P_2(x).$$
\end{cor}

Combining Corollary~\ref{cor3.3} and Corollary~\ref{cor3.4}, we obtain the assertion of Theorem B(i) in the regular case. 
\begin{cor}\label{cor4.6} Suppose that $x$ is regular with respect to $\{C_1,C_{1'},C_2,C_{2'}\}$. 
Then
$$\int_S \ph^o_{KM}(x/\sqrt{2}) =  - \frac14\,\big(\ E_2(C_1,C_2;x) - E_2(C_1,C_{2'};x) -E_2(C_{1'},C_2;x) + E_2(C_{1'},C_{2'};x)\ \big)$$
\end{cor}

\subsection{The irregular case}\label{section-irregular}
Now suppose that $x$ is not regular with respect to the collection $\{C_1,C_2,C_{1'},C_{2'}\}$. 
Fix a vector $y\in V$ so that for all $\e>0$ sufficiently small, $x^\e=x+\e y$ is regular with respect to $\{C_1,C_2,C_{1'},C_{2'}\}$.
Then, by Corollary~\ref{cor4.6}, we have
\beq\label{shifted-x}
\int_S \ph^o_{KM}(x^\e/\sqrt{2}) =  - \frac14\,\big(\ E_2(C_1,C_2;x^\e) - E_2(C_1,C_{2'};x^\e) -E_2(C_{1'},C_2;x^\e) + E_2(C_{1'},C_{2'};x^\e)\ \big).
\eeq
But $\ph_{KM}^o(x)$ is a smooth $2$-form on $D$ which depends smoothly on the vector $x\in V$.  
Also, $E(C,C';x)$ is a smooth function on the space of parameters $\{x,C,C'\}$ where $x\in V$ and 
the negative vectors $C$, $C'$ span a negative $2$-plane. 
This follows from (3.38), (3.23) and Proposition~3.7 (ii) in \cite{ABMP}.
Thus, passing to the limit as $\e\rightarrow 0$ in (\ref{shifted-x}), we obtain the identity of Theorem~B(i) in the irregular case. 
This completes the proof of Theorem~B(i). 

\begin{rem} The point is that, both sides of the identity of Theorem~B(i) are smooth functions of parameters, and so their equality continues to 
hold by continuity on the closure of the set of parameters $\{x; C_1,C_2,C_{1'},C_{2'}\}$ where $x$ is regular with respect to $\{C_1,C_2,C_{1'},C_{2'}\}$. 
\end{rem}

Next note that 
$$E_2(C, C';0) = \Arctan\left( \frac{(C,C')}{\sqrt{\Delta(C,C')}} \right).$$
This follows from \cite{ABMP}, (3.23) together with the fact that $e_2(0,0)=0$ and the boosting procedure. 
As a consequence, we obtain the following nice formula.
\begin{cor}\label{cor4.4}
\begin{align*}
-\int_{S}\ph_{KM}^o(0) &=\frac1{4\pi}\,\int_{S} \O\\
\nass
{}& = \Arctan\left( \frac{(C_1,C_2)}{\sqrt{\Delta_{12}}} \right)-\Arctan\left( \frac{(C_1,C_{2'})}{\sqrt{\Delta_{12'}}} \right) \\
\nass
{}&\qquad\qquad
- \Arctan\left( \frac{(C_{1'},C_2)}{\sqrt{\Delta_{1'2}}} \right) + \Arctan\left( \frac{(C_{1'},C_{2'})}{\sqrt{\Delta_{1'2'}}} \right) .
\end{align*}
\end{cor}

\subsection{An intersection number}
In this section, we give the proof of Proposition~\ref{prop.inter}.
We suppose that $x$ is regular with respect to $\{C_1,C_2,C_{1'},C_{2'}\}$ and that $\P_2(x)\ne 0$, so that, by Lemma~\ref{int-lemma}, 
$D_x\cap S$ consists of a single point $z_0$. If $\zeta \in \OFD$ with $\pi(\zeta)=z_0$ and we write $U(z_0) = z_0^\perp$, then 
$$T_{z_0}(D) \simeq U(z_0)^2, \quad T_{z_0}(D_x) \simeq (U(z_0)\cap x^\perp)^2,$$
and the normal to $D_x$ at $z_0$ is the subspace $(\R x)^2 \subset U(z_0)^2$.  If $e_1, \dots, e_{n-3}$ is a basis for $U(z_0)\cap x^\perp$, 
then the orientation of $D$ is given by the basis vector 
$$[e_1,0]\wedge[0,e_1]\wedge\dots \wedge [e_{n-3},0]\wedge [0,e_{n-3}]\wedge [x,0]\wedge[0,x] \ \in \ \wedge^{2(n-2)} T_{z_0}(D).$$
Note that this is the orientation determined by the complex structure in the hermitian model of $D$. 
If 
$$\phi_*(\frac{\d}{\d s})_{z_0} = [\eta_1,\eta_2], \qquad \phi_*(\frac{\d }{\d t})_{z_0} = [\mu_1,\mu_2],$$
is the tangent frame for the parametrized surface $S$ at the point $z_0$,
then the intersection number of $D_x$ and $S$ at $z_0$ is 
$$I(D_x,S) = \sgn\det\bpm (x,\eta_2)&(x, \eta_2)\\ (x,\mu_2)&(x,\mu_2)\epm.$$
Explicitly, write $B= [B_1(s),B_2(t)]$ and let
$P = -(B,B)$ so that $P\in \Sym_2(\R)$ is positive definite with symmetric square root $P^{\frac12}$. 
Let
$$\und{B} = [\uB_1(s,t),\uB_2(s,t)] = [B_1(s),B_2(t)] \, P^{-\frac12}.$$
Then 
\begin{align*}
\phi_*(\frac{\d}{\d s})_{z_0} &= [-C_1+C_{1'},0] \,P^{-\frac12} + [B_1(s_0),B_2(t_0)] \,\frac{\d}{\d s}\big( \,P^{-\frac12}\,\big)\\
\noalign{hence}
(x,\phi_*(\frac{\d}{\d s})_{z_0} )&= [(x,-C_1+C_{1'}),0] \,P^{-\frac12},\\
\nass
(x,\phi_*(\frac{\d}{\d t}_{z_0} )&= [0,(x,-C_2+C_{2'})] \,P^{-\frac12},
\end{align*}
where we have used the the fact that $(x,B_2(s_0))=(x,B_2(t_0))=0$. This gives
\begin{align*}
I(D_x,S) &= \sgn\det\bpm (x,-C_1+C_{1'})&0\\0&(x,-C_2+C_{2'})\epm\\
\nass
{}&= \sgn((x,C_{1'})-(x,C_1))\, \sgn((x,C_{2'})-(x,C_2))\\
\nass
{}&=\frac14\,[\,\sgn(x,C_1) - \sgn(x,C_{1'})\,][\,\sgn(x,C_2) - \sgn(x,C_{2'})\,],
\end{align*}
as claimed in Proposition~\ref{prop.inter}.

\subsection{Convergence of the holomorphic generating series}\label{hol-conv-section}

Suppose that the collection of negative vectors $\{\,C_1,C_{1'},C_2,C_{2'}\,\}$ satisfies the incidence relations
(\ref{inci-1}), (\ref{inci-2}), and (\ref{inci-2}) and let $S$ be resulting surface in $D$. 
\begin{lem} There exists a positive definite inner product $(\,,\,)_S$ on $V$ such that 
$$(x,x)_z\ge (x,x)_S, \qquad \forall z\in S, \ \forall x\in V.$$
\end{lem}
\begin{proof} Let $(\,,\,)_+$ be any positive definite inner product on $V$ with unit sphere $\mathcal B_+$. 
Define a continuous  function 
$$f: D\lra \R_{>0}, \qquad f(z) = \min_{x\in \mathcal B_+} (x,x)_z.$$
The image $f(S)$ of the compact set $S$ is compact and has a lower bound $\nu\in \R>0$. Take $(\,,\,)_S = \nu (\,,\,)_+$.
\end{proof}

\begin{lem} Suppose that $x\in V$ with $\P_2(x)\ne0$. 
Then 
$$(x,x) \ge (x,x)_S.$$
\end{lem}
\begin{proof}
By Lemma~\ref{int-lemma}, the condition, $\P_2(x)\ne0$ implies that $D_x\cap S$ is a single point $z_0$. 
Thus $R(x,z_0)=0$ and we have 
$$(x,x) = (x,x)_{z_0} -2 R(x,z_0) = (x,x)_{z_0} \ge (x,x)_S,$$
as claimed. 
\end{proof}

\begin{prop} Assume that the collection of negative vectors $\{\,C_1,C_{1'},C_2,C_{2'}\,\}$
satisfies the incidence conditions (\ref{inci-1}), (\ref{inci-2}),  and (\ref{inci-3}). 
Then the series
$$\sum_{x\in L+\mu} \P_2(x)\,q^{\frac12(x,x)}$$
is termwise absolutely convergent. 
\end{prop}
\begin{proof} By the previous lemma, this series is dominated by the convergent series
$$\sum_{x\in L+\mu} \exp(-\pi v (x,x)_S).$$
\end{proof}

\begin{rem} This result shows that our incidence conditions 
are sufficient to 
establish the convergence obtained from conditions (4.2), Theorem~4.1 in \cite{ABMP}.
\end{rem}

\section{Additional remarks}

\subsection{Shadows}\label{section-shadows}
It is natural to consider the image of the non-holomorphic modular form $I_\mu(\tau;S)$ under the lowering operator
$L_k = -2i v^2\frac{\d}{\d \bar\tau}$ or, equivalently, under the $\xi$-operator.  This is done in \cite{ABMP}.  In this section, we explain how this image can be 
calculated  
using our theta integrals. 

We begin with the identity
$$-2i\,v^2\frac{\d}{\d \bar\tau}\{\,\ph_{KM}(\tau,x)\,\} = v^2\,q^{\frac12(x,x)} d\big(\ \frac{\d}{\d v}\{\psi(x\sqrt{v})\}\ \big).$$
Recalling (\ref{def-psi}), on $\OFD$ we have
\begin{align*} 
v^2\,\frac{\d}{\d v}\{\psi(x\sqrt{v})\}&=  v^2\,e^{-2\pi v R}\,\big(\ (\,(x,\zeta_1)(x,d\zeta_2)-(x,\zeta_2)(x,d\zeta_1) \,) - R\,(\zeta_2,d\zeta_1)\ \big)
\end{align*}
and its restriction to the connection subspace can be written as
$$v^2\frac{\d}{\d v}\{\psi(x\sqrt{v})\} = v\,\Psi_M^o(x\sqrt{v}),$$
where
\begin{align*} 
\Psi_M^o(x)&=  e^{-2\pi R}\,\big(\ (\,(x,\zeta_1)(x,d\zeta_2)-(x,\zeta_2)(x,d\zeta_1) \,)\ \big).
\end{align*}
Let
$$\Psi_M(\tau,x) = q^{\frac12(x,x)}\,\Psi_M^o(x\sqrt{v}),$$
a Schwartz function on $V$ valued in $1$-forms on $D$. 
Define the associated theta form 
$$\theta_\mu(\tau,\Psi_M) = \sum_{x\in L+\mu} \Psi_M(\tau,x).$$
Then
$$-2i\,v^2\,\frac{\d}{\d \bar\tau}\{\,\theta_\mu(\tau,\ph_{KM})\,\} =v\,d\, \theta_\mu(\tau,\Psi_M),$$
and hence the image of $I_\mu(\tau;S)$ under lowering can be written as an integral over the boundary of $S$,
\begin{align}\label{lower-theta-integral}
L_{\frac{n}2}I_\mu(\tau;S)=-2i\,v^2\frac{\d}{\d \bar\tau}\{\,I_\mu(\tau;S)\,\} 
{}&=v\,\int_{\d S} \theta_\mu(\tau,\Psi_M).
\end{align}

We could, of course, now
integrate termwise using the parametrization of $\d S$ as before and arrive at the formulas for the shadow given in \cite{ABMP}.
However, it is more enlightening to continue using the geometry. Recall that $\d S= \gamma_1+\gamma_{2'}-\gamma_{1'}-\gamma_2$.

We note that, by construction, the path $\gamma_1$ is the image in $D$ of a geodesic in the hyperbolic space 
$$D(V_1)=\{\,\nu\in V_1\mid (\nu,\nu)=-1\,\}$$
in $V_1$ under the inclusion
$$j_1:D(V_1)\lra D, \qquad \nu \ \mapsto\ \sspan\{\und{C_1},\nu\}_{\text{p.o.}}.$$
For the  
decomposition 
$$V = \R\,C_1 + V_1,$$
we write 
$$x = x_0 + x_{1},\qquad x_0=-(x,\und{C}_1)\,\und{C}_1,$$
and we note that 
$$R(x,j_1(\nu)) = (x,\und{C}_1)^2 + R_{1}(x_1,\nu),$$
where
$$R_{1}(x_{1},\nu) = (x_{1},\nu)^2.$$

\begin{lem} Let $\ph_{KM}^{(n-2,1)}(\tau,x_1)$ be the Schwartz form associated to the space $V_1$ of signature $(n-2,1)$. 
It is a closed $1$-form on $D(V_1)$ whose value on a tangent vector $\eta\in T_{\nu}(D(V_1))$ is
$$\ph_{KM}^{(n-2,1)}(\tau,x_1)(\eta) = v^{\frac12}(x_1,\eta)\, q^{\frac12(x_1,x_1)}\, e^{-2\pi v R_1(x_1,\nu)} .$$
Then
\beq\label{inductive-formula}
v\,j_1^*\Psi_M(\tau,x)  =  v^{\frac32}\,(x_0,\und{C}_1)\,(\bar{q})^{-\frac12(x_0,x_0) }\, \ph_{KM}^{(n-2,1)}(\tau,x_1).
\eeq
\end{lem}
\begin{proof}
If $\nu\in D(V_1)$,  there is an identification
$$T_\nu(D(V_1)) = \{ \eta\in V_1\mid (\nu,\eta)=0\}.$$
The image of the tangent vector $\eta\in T_\nu(H_1)$ 
in $T_{j_1(\nu)}(D)$ is the vector $[0,\eta]$ where we have used the point $[\und{C}_1,\nu]\in \OFD$ in the identification
of $T_{j_1(\nu)}(D)\simeq U(j_1(\nu))^2$.
Then 
\begin{align*} 
j_1^*\Psi_M(\tau,x)(\eta)= q^{\frac12 (x_0,x_0) +\frac12 (x_1,x_1)}\,e^{-2\pi v [(x,\und{C}_1)^2+R_1(x_1,\nu)]}\,v^{\frac12}(x,\und{C}_1)\,v^{\frac12} (x_1,\eta).
\end{align*}

\end{proof}
Now assume that $C_1$ is a rational vector and let 
$$L_0=L\cap \Q \,C_1,\qquad L_1= L\cap V_1,$$ 
so that 
$$L_0+L_1\subset L \subset L^\vee\subset L_0^\vee + L_1^\vee.$$
Then, we can write any coset as 
$$\mu+L = \bigsqcup_{\l = \l_0+\l_1} (\l_0+L_0)+(\l_1+L_1),$$
and we have
$$v\,j_1^*\theta_\mu(\tau,\Psi_M) = \sum_{\l} v^{\frac32}\,\overline{ \theta_{0,\l_0}(\tau)}\,\theta_{\l_1}(\tau,\ph_{KM}^{(n-2,1)}).$$
Here 
$$\theta_{0,\l_0}(\tau) = \sum_{x_0\in \l_0+L_0} (x_0,\und{C}_1)\, q^{-\frac12(x_0,x_0)}$$
is the theta series of weight $\frac32$ attached to the coset $\l_0+L_0$ and the positive definite unary form $-(x_0,x_0)$. 
The form $v^{\frac32}\,\overline{\theta_0(\tau)}$ then has weight $-3/2$, while $\theta_{\l_1}(\tau,\ph_{KM}^{(n-2,1)})$ has weight $(n-1)/2$ so the product 
has weight $n/2 -2$, as required. 

But now
\begin{align}\label{inductive-formula}
v\,\int_{\gamma_1} \theta_\mu(\tau,\Psi_M) & =  \sum_{\l} v^{\frac32}\,\overline{ \theta_{0,\l_0}(\tau)}\, 
\int_{\gamma(C_{2\perp1},C_{2'\perp1})}\theta_{\l_1}(\tau,\ph_{KM}^{(n-2,1)}),
\end{align}
where the right side is a linear combination of products of weight $-\frac32$ antiholomorphic unary theta series times the Zwegers' indefinite theta series
which arise as integrals of the $\theta_{\l_1}(\tau,\ph_{KM}^{(n-2,1)})$'s along the geodesic $\gamma(C_{2\perp1},C_{2'\perp1})$ in $D(V_1)$. 
The remaining terms in the expression (\ref{lower-theta-integral}) for $L_{n/2} I_\mu(\tau;S)$ as an integral over $\d S$ 
have the same form.
This formula reveals the inductive nature of the shadow.

\subsection{Other signatures}\label{section-other-sigs}
The case in which $V$ has signature $(n-r,r)$ is considered in section~6 of \cite{ABMP}.  As remarked in the introduction, the relevant Schwartz forms 
$\ph_{KM}^{(n-r,r)}(x)$ are considered in \cite{KM.I}, \cite{KM.II} and \cite{KM.IHES}.  In the present section, we indicated the first steps in generalizing 
the constructions of earlier sections, leaving details for another occassion. 

Let $D$ be the space of oriented negative $r$-planes in $V$.  

We suppose that a collection of $r$ pairs of negative vectors $\{C_1,C_1'\}, \dots, \{C_r,C_r'\}$ is given with the following incidence properties. 
For a subset $I\subset \{1, \dots,r\}$, let $C_I$ be the ordered set of $r$ vectors where we take $C'_j\in C_I $ if $j\in I$ and $C_j\in C_I$ if $j\notin I$
and the vectors are ordered according to the index $j$. 
Thus, $C_{\emptyset} = \{C_1,\dots,C_r\}$, etc.  
We require
\begin{itemize}
\item[(i)]
Each collection $C_I$ spans a negative $r$-plane
$$P_I = \sspan\{C_I\}.$$
\item[(ii)]
These planes are distinct, so that there are $2^r$ of them. 
\item[(iii)] The oriented negative $r$-planes 
$$z_I = \sspan\{C_I\}_{\text{p.o.}}$$
all lie on the same component $D^+$ of $D$. 
\end{itemize}
We will not attempt express these in terms of determinants of minors of Gram matrices, although this is clearly possible. 
Note that the conditions can be expressed inductively. 
We use the notation $C_{j'}=C_j'$, as before. For example, for $r=3$, the collection of vectors 
$\{C_{2\perp1},C_{2'\perp1},C_{3\perp1}, C_{3'\perp1}\}$ 
should satisfy the incidence relations analogous to (\ref{inci-1}), (\ref{inci-2}) and (\ref{inci-3}), and similarly for the other $5$ `projections'.
Thus, the set of $2^r$ points $\{z_I\}$ form the vertices of a geodesic hypercube $S$ in $D$. 
This cube can be parametrized by
$$\phi:[0,1]^r \isoarrow S\subset D, \qquad s=[s_1,\dots,s_r] \mapsto \sspan\{B_1(s_1),\dots,B_r(s_r)\},$$
where 
$$B_j(s_j) = (1-s_j)C_j + s_j C_{j'}.$$

For a vector $x\in V$, let 
$$\P_r(x)= \frac{1}{2^r}\,[\,\sgn(x,C_1) - \sgn(x,C_{1'})\,]\dots [\,\sgn(x,C_r) - \sgn(x,C_{r'})\,],$$
as in (6.5) of \cite{ABMP}.  Lemma~\ref{int-lemma} and the arguments of section~\ref{hol-conv-section} carry over immediately and
yield the following. Here 
$$D_x = \{ \, z\in D\mid x\perp z\ \}$$
is an oriented\footnote{The orientation is defined as follows: Suppose that an orientation of $D$ has been fixed. 
Then, at a point $z\in D_x$, we have $T_z(D)= T_z(D_x)+ N_z(D_x)$ where, under the identification $T_z(D) = \Hom(z,z^\perp)$, 
the normal space $N_z(D_x) = \Hom(z,\R x)$. Since $z\in D$ is oriented, the basis vector $x$ in $\R x$ determines an orientation of 
$N_z(D_x)$ and hence of $T_z(D_x)$. } totally geodesic submanifold of $D$ of codimension $r$ and 
$L\subset L^\vee$ is an integral lattice in $V$. 

\begin{prop} (i) Suppose that $\P_r(x)\ne 0$. Then $D_x\cap S$ consists of a single point. \hfb
(ii) The $q$-series
\beq\label{hol-sig-r}
\sum_{x\in L+\mu} \P_r(x)\, e^{\frac12(x,x)},\qquad \mu\in L^\vee/L,
\eeq
is termwise absolutely convergent. 
\end{prop}

The theta form 
$$\theta_\mu(\tau,\ph_{KM}) = \sum_{x\in L+\mu} \ph_{KM}(\tau,x)$$
is a closed $r$-form on $D$ and we define
$$I_\mu(\tau;S) = \int_S \theta_\mu(\tau,\ph_{KM}).$$
As before $I_\mu(\tau;S)$ is a non-holomorphic modular form of weight $n/2$ and transformation law inherited from that of 
$\theta_\mu(\tau;\ph_{KM})$. 
We expect that $I_\mu(\tau;S)$ is again the modular completion of the $q$-series (\ref{hol-sig-r}).
A solution of the equation
$$\frac{\d}{\d \bar\tau}\ph_{KM}(\tau,x) = d \Psi_M(\tau,x)$$
is defined in section~8 of \cite{KM.IHES} and the analogue of the form $\psi(x)$ used above can be obtained from it. 
It remains to do the boundary calculations explicitly and to determine the relation with the generalized error 
functions defined in section 6 of \cite{ABMP}.  In particular, it seems likely that there will be an analogue of the 
inductive relation (\ref{inductive-formula}) for the shadow in general. We hope to return to this in the future. 

\section{Appendix:  Differential forms on $D$}

In this section, we review the conventions used above concerning tangent spaces and differential forms. 
The basic idea is to simplify computations by pulling back to the frame bundle and then viewing this bundle as an open subset of $V^2$. 

\subsection{Tangent vectors}\label{subsec-tan-vecs}

We use the following construction.  Let
$$\FD = \{ \, \zeta = [\zeta_1,\zeta_2] \in V^2\mid (\zeta, \zeta)<0\ \}$$
and
$$\OFD = \{ \, \zeta = [\zeta_1,\zeta_2] \in V^2\mid (\zeta, \zeta)=-1_2\ \}$$
be the spaces of oriented negative frames and oriented orthogonal negative frames respectively.
We have natural projections to $D$ sending a frame to the $2$-plane it spans.  
Note that $\FD \subset V^2$ is an open subset so that there is a natural identification $T_\zeta(\FD) = V^2$. 
Also, $\FD$ (resp. $\OFD$)  is a $\GL_2(\R)^+$ (resp. $\SO(2)$) torsor over $D$.
For convenience, for $z\in D$ we write $U(z) = z^\perp$.  If $\zeta\in \FD$ with span $\pi(\zeta)=z$,  we obtain 
an isomorphism
\beq\label{tan.id.FD}
d\pi_{\zeta} : U(z)^2 \isoarrow T_z(D).
\eeq
satisfying the equivariance condition
$$
\xymatrix{ U(z)^2\ar[r]^{d\pi_{\zeta}}\ar[d]^{R_g}& T_z(D) \ar[d]\\
U(z)^2\ar[r]^{d\pi_{\zeta'}}&T_z(D),
}
$$
where $g\in \GL_2(\R)^+$ and $\zeta'= \zeta \,g$. 
Thus, the spaces $U(z)^2$ define a connection on the $\GL_2(\R)^+$ principal bundle $\FD$
and provide a convenient way to describe differential forms on $D$.   
Note that,  for a choice of $\zeta$ with $\pi(\zeta)=z$,  the natural isomorphism
\beq\label{tan.natural}
\Hom(z, z^\perp) \isoarrow T_z(D),
\eeq
can be written as the composition of the isomorphism
$$\Hom(z, z^\perp) \isoarrow U(z)^2, \qquad \l \mapsto [\l(\zeta_1),\l(\zeta_2)],$$
given by the basis $\zeta$, with the map (\ref{tan.id.FD}). 

For a point $\zeta\in \OFD \subset \FD$, we have
$$V^2=T_{\zeta}(\FD) \ \supset \ T_{\zeta}(\OFD) = \{\,\eta=[\eta_1,\eta_2] \mid (\eta,\zeta)+(\zeta,\eta)=0\,\},$$
i.e., $(\eta_1,\zeta_1)=0$, $(\eta_2, \eta_2)=0$, and $(\eta_1,\zeta_2) = -(\eta_2,\zeta_1)$. 
This contains the connection subspace $U(z)^2$ and the additional vertical subspace $\R \,[\zeta_2,-\zeta_1] \subset z^2$. 

Suppose that $\gamma:[0,1]\lra D$ is a smooth curve given explicitly by a lift
$$\gamma(s) = \sspan\{\zeta(s)\}, \qquad \zeta(s) = [\zeta_1(s),\zeta_2(s)],$$
to a smooth curve in $\OFD$, i.e., with $(\zeta(s), \zeta(s))= -1_2$. 
The vector field along the lift $\zeta(s)$ obtained by pushing forward $\frac{d}{ds}$ is given by
$$d\zeta_*(\frac{d}{ds})= [\dot\zeta_1(s), \dot\zeta_2(s)],$$
and the image along $\gamma$ is 
$$d\gamma_*(\frac{d}{ds})=d\pi_{\zeta}( [\dot\zeta_1(s), \dot\zeta_2(s)]).$$
Under the isomorphism (\ref{tan.natural}), this corresponds to the vector
$$\zeta_1\tt \pr_{U(z)}\dot\zeta_1(s) + \zeta_2\tt \pr_{U(z)}\dot \zeta_2(s)\ \in \ z\tt z^\perp,$$
where $\pr_{U(z)}$ is the orthogonal projection of $V$ to $U(z)$. \hfb 
\begin{rem}  There is one subtle point here. The condition $(\zeta(s), \zeta(s))= -1_2$
implies that 
$$0= (\dot\zeta(s),\zeta(s)) + (\zeta(s), \dot\zeta(s)). $$
But the component $(\dot\zeta_1(s),\zeta_2(s))$ need not be zero and so the $\dot\zeta_j(s)$'s may not lie in $U(z)$, 
hence the need for the projection. On the other hand, if we integral $\pi^*(\psi(x))$ along the lift of a curve to $\OFD$, then the pullback 
will vanish on the vertical component of the tangent vector, and so we do not need to take any projection. 
\end{rem}

\begin{rem} It will be convenient to compute differential forms on $D$ by working with their pullbacks to $\OFD$. 
In the end, we 
obtain forms on $D$ by restricting to the connection subspace. This restriction cannot, of course, be done at the various 
intermediate steps!
\end{rem}

\subsection{Proof of Proposition~\ref{dpsi}}\label{subsec-proof-dpsi}

On $\OFD$, we have 
\begin{align}
\psi(x)  &= -\frac{1}{2\pi} e^{-2\pi R}\,\big(\ R^{-1}(\,(x,\zeta_1)(x,d\zeta_2)-(x,\zeta_2)(x,d\zeta_1) \,) - (\zeta_2,d\zeta_1)\ \big).
\end{align}
where we have omitted the $\pi^*$.  Note that since $(\zeta,\zeta)=-1_2$ on $\OFD$, we have 
$$R = (x,\zeta)(\zeta,x),\qquad dR = 2(x,\zeta)(d\zeta,x) = 2(x,\zeta_1)(x,d\zeta_1) + 2 (x,\zeta_2)(x,d\zeta_2),$$
and relations $(\zeta_1,d\zeta_1)=0$, $\zeta_2,d\zeta_2)=0$, and  $(\zeta_2,d\zeta_1) = -(\zeta_1,d\zeta_2)$.
Thus, we have 
$$dR \wedge (\zeta_2,d\zeta_1) = -dR\wedge (\zeta_1,d\zeta_2)=0.$$
Now we calculate
\begin{align*}
d\psi(x)&= e^{-2\pi R}\,[\,R^{-1}+ \frac1{2\pi}\,R^{-2}\,]\,dR \wedge((x,\zeta_1)(x,d\zeta_2)-(x,\zeta_2)(x,d\zeta_1) \,)\\
\nass
{}&\qquad\quad -\frac1{2\pi}\,e^{-2\pi R}\,R^{-1}\,2\, (x,d\zeta_1)\wedge (x,d\zeta_2)\\
\nass
{}&\qquad\qquad \quad -\frac1{2\pi}\,e^{-2\pi R} \,(d\zeta_1,d\zeta_2).
\end{align*}
But we have
\begin{align*}
dR \wedge(\,(x,\zeta_1)(x,d\zeta_2)-(x,\zeta_2)(x,d\zeta_1)\,)&= 2 R\,(x,d\zeta_1)\wedge(x,d\zeta_2).
\end{align*}
so that 
\begin{align*}
d\psi(x)&=2\,\big[\,(x,d\zeta_1)\wedge(x,d\zeta_2) -\frac1{4\pi}\,(d\zeta_1,d\zeta_2)\,\big]\,e^{-2\pi R},
\end{align*}
as claimed. 

\subsection{The generalized error function}
For completeness, in this section, we record the definitions of the generalized error function and its `boosted' version
given in \cite{ABMP}. 

The classical error function and its complementary version, in the normalization of \cite{ABMP}, are given by 
\begin{align*}
E_1(u) & = \sgn(u) \,\Erf(|u|\sqrt{\pi}) = 2 u \int_0^1 e^{-\pi u^2 v^2}\, dv \\
\nass
M_1(u)&= -\sgn(u) \,\Erfc(|u|\sqrt{\pi})= -\frac{2}{\pi} \sgn(u)\int_0^\infty e^{-\pi u^2(v^2+1)}\,(v^2+1)^{-1}\,dv.
\end{align*}
The auxilliary function used to define the generalized error function is then 
\begin{align*}
\nass
e_2(u_1,u_2)&= 2 u_2 \int_0^1 e^{-\pi t^2 u_2^2}\,E_1(t u_1)\,dt \\
\nass
{}&= 4 u_1 u_2 \int_0^1\int_0^1 e^{-\pi t^2( u_2^2 + u_1^2 v^2)}\, t \,dt \,dv\\
\nass
{}&=  \frac{2 u_1u_2}{\pi}\int_0^1 \big(\,1- e^{-\pi (u_2^2+u_1^2 v^2)} \ \big)\,(u_2^2+u_1^2v^2)^{-1}\,dv\\
\nass
{}&= \frac{2}{\pi} \,\Arctan(\frac{u_1}{u_2}) - \frac{2}{\pi} u_2\,e^{-\pi u_2^2}\,\int_0^{u_1} e^{-\pi v^2}\,(u_2^2+v^2)^{-1}\,dv\\
\nass
{}&= \frac{2}{\pi} \,\Arctan(\frac{u_1}{u_2}) - \tilde e_2(u_1,u_2),
\end{align*}
where the last line defines $\tilde e_2(u_1,u_2)$, cf. (\ref{e2-int}).
The generalized error function $E_2$ can be expressed as, cf. (3.24) of \cite{ABMP}, 
\beq\label{our-def-E2}
E_2(\a;u_1,u_2) = -\tilde e_2(u_1,u_2) - \tilde e_2(u_1',u_2') + \sgn(u_2)\,\sgn(u_2'),
\eeq
where
$$u_1'= \frac{u_2-\a u_1}{\sqrt{1+\a^2}}, \qquad u_2'=  \frac{u_1+\a u_2}{\sqrt{1+\a^2}}.$$
It has, of course, a more elegant definition given in section 3.1 of \cite{ABMP}, but (\ref{our-def-E2}) is all we will need. 
Now the `boosted' version, Definition~3.14 of \cite{ABMP}, is defined for a pair of negative vectors $C_1$, $C_2$ 
with $\Delta(C_1,C_2)>0$ and a vector $x\in V$ by
$$E_2(C_1,C_2;x) = E_2\left(\frac{(C_1,C_2)}{\sqrt{\Delta(C_1,C_2)}}; \frac{(C_{1\perp 2},x)}{\sqrt{(C_{1\perp2},C_{1\perp2})}},
\frac{(C_{2},x)}{\sqrt{(C_{2},C_{2})}}\right).$$
After a short calculation using (\ref{our-def-E2}), this yields (\ref{Etoe2}).

\end{document}